\documentclass{mcom-l}

\newtheorem{thm}{Theorem}[section]
\newtheorem{ap}{Assumption}[section]

\newtheorem{prop}[thm]{Proposition}

\theoremstyle{remark}
\newtheorem{rk}[thm]{Remark}
\newtheorem{ex}[thm]{Example}

\numberwithin{equation}{section}

\usepackage{graphicx} 
\usepackage{subfigure}
\usepackage{extarrows}
\usepackage{latexsym}
\usepackage{amsmath}
\usepackage{amssymb}
\usepackage{amsfonts}
\usepackage{verbatim}
\usepackage{mathrsfs}
\usepackage[modulo]{lineno} %
\usepackage[latin1]{inputenc}
\usepackage{color}
\usepackage{xcolor}
\usepackage[colorlinks,citecolor=blue,urlcolor=blue]{hyperref}
\usepackage{soul}%
\usepackage{enumitem}
\usepackage{amsthm}

\newcommand{\ee}{\mathbb E}

\newcommand{\pp}{\mathbb P}
\newcommand{\nn}{\mathbb N}
\newcommand{\rr}{\mathbb R}

\newcommand{\1}{\mathbf 1}

\newcommand{\BB}{\mathcal B}
\newcommand{\CC}{\mathcal C}

\newcommand{\LL}{\mathcal L}

\newcommand{\OOO}{\mathscr O}
\newcommand{\FFF}{\mathscr F}

\newcommand{\PPP}{\mathscr P}

\newcommand{\dif}{\mathrm{d}}

\newcommand{\<}{\langle}
\renewcommand{\>}{\rangle}
\allowdisplaybreaks \allowdisplaybreaks[4]

\newcommand{\dd}{\mathrm{d}}

\newcommand{\abs}[1]{\left\lvert #1 \right\rvert}
\newcommand{\norm}[1]{\left\lVert #1 \right\rVert}

\begin{document}

\title[Ergodic Estimates of One-Step Approximations for Superlinear SODEs]
{Ergodic Estimates of One-Step Numerical Approximations for Superlinear SODEs}
 
\author{Xin LIU}
\address{Department of Mathematics, Southern University of Science and Technology, Shenzhen 518055, China}
\email{12532002@mail.sustech.edu.cn}
 
\author{Zhihui LIU}
\address{Department of Mathematics \& National Center for Applied Mathematics Shenzhen (NCAMS) \& Shenzhen International Center for Mathematics, Southern University of Science and Technology, Shenzhen 518055, China}
\email{liuzh3@sustech.edu.cn (Corresponding author)}
 
\thanks{The authors are supported by the National Natural Science Foundation of China, No. 12101296, Guangdong Basic and Applied Basic Research Foundation, No. 2024A1515012348, and Shenzhen Basic Research Special Project (Natural Science Foundation) Basic Research (General Project), Nos. JCYJ20220530112814033 and JCYJ20240813094919026.}


\date{}



\maketitle

\begin{abstract}
This paper establishes the first-order convergence rate for the ergodic errors and ergodic limits of numerical approximations to a class of stochastic ODEs (SODEs) with superlinear coefficients and multiplicative noise. By leveraging the generator approach within the Stein method, we derive a general error representation formula for one-step numerical schemes. Our framework applies to several recently studied schemes, including the tamed Euler, projected Euler, and backward Euler methods. Numerical experiments confirm the optimality of the derived ergodic estimates.
 \end{abstract}

\section{Introduction}

Consider the following $d$-dimensional SODE driven by an $m$-dimensional Wiener process $W$ in a complete filtered probability space  $(\Omega, \FFF, \{\FFF_t\}_{t \ge 0}, \pp)$:
 \begin{align}
\label{sde} 
    \dd X_t
    =
    b(X_t) 
  \, \dd t
    +
    \sigma(X_t) 
    \, \dd W(t), 
    \quad t >0. 
\end{align}
Here $b: \rr^d \rightarrow \rr^d$ and  $\sigma: \rr^d \rightarrow \rr^{d \times m}$ are measurable superlinear functions satisfying certain monotone and polynomial growth conditions.

Under appropriate conditions, Eq. \eqref{sde} admits a unique invariant measure $\pi$.
Explicitly characterizing $\pi$ is generally infeasible and is known only for specific cases, such as the gradient Langevin equation with additive noise or one-dimensional SODEs with multiplicative nondegenerate noise \cite{LL25+}. 
Consequently, constructing accurate numerical approximations of $\pi$ is a problem of significant interest. 
The study of long-time approximations for SODEs is of growing importance, with substantial applications in scientific computing, statistics, machine learning, and generative AI. For instance, sampling in score-based diffusion models -- a key technique in generative AI -- fundamentally relies on long-time SODE approximations. 
 
The Euler--Maruyama (EM) scheme is effective for Lipschitz coefficients (see, e.g., \cite{FSX19, RT96}) or Markovian switching diffusion (see, e.g., \cite{Sha18}). 
However, for coefficients with superlinear growth, the EM scheme would blow up in the $p$-th moment for all $p \ge 2$ \cite{HJK11}, which precludes ergodicity. 
As a result, several modified schemes have been developed to preserve ergodicity, including the backward Euler method (BEM) \cite{LMW23, MSH02, Tal02}, the stochastic theta method \cite{LL25}, and modified Euler methods (MEMs) such as the tamed Euler method (TEM) and projected Euler method (PEM) \cite{LWWZ25, LW24}; see also \cite{LMYY18} for Markovian switching diffusions.

A central question is to quantify the ergodic error $|\pi(\varphi)-\pi_\tau(\varphi)|$ between the numerical invariant measure $\pi_\tau$ and the exact invariant measure $\pi$ for a given step-size $\tau \in (0, 1)$ and a test function $\varphi$.
Using the generator approach \cite{Bar90} within the Stein method \cite{Ste72}, the authors of \cite{FSX19} derived an ergodic error representation for additive noise cases. They obtained a convergence rate of $\OOO(\tau^{1/2} |\ln \tau|)$ for the EM scheme (the unadjusted Langevin algorithm in the Langevin equation case) with Lipschitz drift and nondegenerate additive noise, where $\varphi$ is a Lipschitz function. This result was later generalized in [4] to dissipative Lipschitz systems with nondegenerate multiplicative noise and jump diffusions, removing the logarithmic factor for test functions in $\CC_b^3(\rr^d)$. We also refer to \cite{PP23} and the references therein for a nonasymptotic bound analysis of the decreasing-step EM scheme. 
 
In practical applications, a more natural goal is to approximate the ergodic limit, i.e., the spatial average $\pi(\phi)$ for a given test function $\phi$, using time-averaging estimators of the form $\frac{1}{N}\sum_{k=0}^{N-1}\mathbb E\varphi(Y_k)$.
Such estimators can be effectively computed via simulation. One then aims to characterize the error $|\frac{1}{N}\sum_{k=0}^{N-1}\mathbb E\varphi(Y_k)-\pi(\phi)|$ with respect to the number of steps $N$ and the step-size $\tau$. 
The authors of \cite{MST10} derived error estimates for time-averaging estimators of numerical schemes, based on the associated Stein equation and an assumption of local weak convergence order, for Eq. \eqref{sde} with Lipschitz coefficients on a torus (so that all coefficients are uniformly bounded). The superlinear case remains unknown.
Our representation of the ergodic error $|\pi(\varphi)-\pi_\tau(\varphi)|$ leads to a sharp error estimate for the ergodic limit, provided that geometrical ergodicity holds for the general class of one-step numerical approximations considered.

Since the weak convergence rate is typically twice the strong convergence rate, an ergodic error of order $\OOO(\tau)$ and an ergodic limit error of order $\OOO(\tau+N^{-1})$ are expected.
This paper achieves this optimal first-order rate for a family of superlinear SODEs with multiplicative noise and a class of one-step numerical approximations, assuming test functions belong to $\CC_b^4(\rr^d)$. 
 
To derive the above optimal convergence rates for the ergodic error and ergodic limit for a general one-step approximation, we first derive the Stein equation associated with the invariant measure of Eq. \eqref{sde}, as in \cite{FSX19}, and present one solution of this equation in terms of the Markov semigroup and the invariant measure (see Theorem \ref{tm-ste}). 
 
By exploiting the stationary property of invariant measures, we establish a relation between the discrete generator of the numerical approximations and the generator of Eq. \eqref{sde}. 
This allows us to derive a new representation of the ergodic error for a general class of one-step approximations to Eq. \eqref{sde} (see Theorem \ref{tm-err}).  
This error representation consists of two parts: the first pertains to the error induced by the choice of initial datum, and the second concerns the remainder error arising from the continuous interpolation process associated with the one-step approximations. An error estimate for the ergodic limit follows immediately, provided that geometrical ergodicity holds for these approximations.

Finally, we demonstrate that a general one-step approximation naturally fits into the aforementioned framework. In particular, we establish the optimal first-order convergence of the ergodic error for a sequence of numerical approximations, including the recently investigated TEM, PEM, and BEM, applied to the superlinear SODE \eqref{sde} driven by multiplicative noise. 
This analysis requires boundedness of the fourth derivatives for the solution of the Stein equation, which is ensured by the integrability in $\rr_+$ of the fourth derivatives for the solution of Eq. \eqref{sde} with respect to the initial datum. 

The paper is organized as follows. Section \ref{sec2} presents the main ideas, including the Stein equation and the general ergodic error representations. Section \ref{sec3} applies these results to superlinear SODEs and provides detailed analyses for specific numerical schemes. Numerical experiments verifying the optimality of these ergodic estimates are presented in the final section.

\section{Main Idea and Results}
\label{sec2}

This section introduces frequently used notations, the Stein equation and its solution, and the main result on the ergodic error for general one-step approximations.

\subsection{Notations}

Let us begin with some frequently used notations.  
Throughout, $\abs{\cdot}$, $\<\cdot, \cdot\>$, and $\norm{\cdot}_{HS}$ denote the Euclidean norm, the inner product in $\rr^d$, and the Hilbert--Schmidt norm in $\rr^{d \times m}$, respectively. 

We use $\BB_b(\rr^d)$ and ${\rm Lip}(\rr^d)$ to denote the Banach spaces of all bounded measurable and Lipschitz continuous mappings $\phi:\rr^d \rightarrow \rr$. 
 Moreover, for $k \in \nn_+$, let $C_b^k(\rr^d)$ be the Banach space of all $k$-times continuously differentiable functions with bounded partial derivatives $\nabla^i\varphi(x) \in \LL((\rr^d)^{\otimes i}; \rr)$, the $i$-fold linear functional in $\rr^d$ with $1\le i\le k$, endowed with the seminorm 
 $$|\varphi|_k:=\sum_{i=1}^k \sup_{x\in\rr^d} \|\nabla^i\varphi(x)\|_{\LL((\rr^d)^{\otimes i}; \rr)}.$$

\subsection{Stein Equation}

Assume Eq. \eqref{sde} has a unique continuous strong solution $(X_t^\cdot)_{t \ge 0}$ with Markov semigroup $(P_t)_{t \ge 0}$ given by
 \begin{align} \label{pt}
P_t \varphi(x)=\ee \varphi(X_t^x), 
\quad x\in\mathbb R, ~ \varphi \in \mathcal B_b(\rr^d), ~ t \ge 0.
\end{align}
Assume that $(X_t^\cdot)_{t \ge 0}$ has bounded algebraic moments in the sense that 
\begin{align*}
\sup_{t \ge 0} \ee |X_t^x|^p \le C(1+|x|^p), \quad p \ge 2, ~x \in \rr^d,
\end{align*}
and that $\ee \int_0^T |\sigma_j(X_t^\cdot)|^2 \dd t<\infty$ for all $T>0$ and $j=\{1,2,\cdots,m\}$, ensuring that $\int_0^\cdot \sigma(X_s^\cdot) \dd W(s)$ is a martingale.
Note that the definition \eqref{pt} is also valid for $\varphi \in {\rm Lip}(\rr^d)$.
Assume that $(P_t)_{t \ge 0}$ has a unique invariant measure $\pi$ with bounded algebraic moments and the mixing property:  
 \begin{align} \label{mix}
\lim_{t \to+\infty}P_t \varphi(x)=\pi(\varphi)=:\int_{\rr^d} \varphi(x) \pi(\dd x), \quad x\in\mathbb R, ~ \varphi \in \mathcal B_b(\rr^d).
\end{align}

For any $f \in C^2(\rr^d)$, It\^o formula gives 
 \begin{align}\label{ito}
\dif f(X_t) & =\mathcal Af(X_t) \dif t +\nabla f(X_t)\sigma(X_t)\dif W(t), \quad t > 0,
	\end{align}
where $\mathcal A$ is the infinitesimal generator of Eq. \eqref{sde}:
	\begin{align} \label{df-a}
		\mathcal Af(x):=\nabla f(x)b(x)+\frac{1}{2} \sum_{j=1}^m \nabla^2 f(x) (\sigma_j(x), \sigma_j(x)), \quad x \in \rr^d.
	\end{align}
From \eqref{ito} we obtain
	\begin{align} \label{af}
		\ee [f(X_t)-f(X_0)]=\int_0^t \ee \mathcal A f(X_s)\dif s,
		\quad t \ge 0.
	\end{align}
	If $X_0 \sim \pi$, then $\int_0^t \ee [\mathcal Af(X_s)]\dif s=0$ for any $t\geq 0$. 
	By the continuity of $X$, this implies $\ee [\mathcal A f(X_0)]=0$.
	
	Given an $\varphi \in \BB_b(\rr^d)$ or ${\rm Lip}(\rr^d)$, let us consider the Stein equation associated with the invariant measure $\pi$ of Eq. \eqref{sde}:
	\begin{align}\label{ste}
		\varphi(x)-\pi(\varphi)=\mathcal A f_\varphi(x), \quad x \in \rr^d.
	\end{align} 	
The following theorem constructs a solution to the Stein equation \eqref{ste}.

  \begin{thm} \label{tm-ste}
For any $\varphi \in \BB_b(\rr^d)$ or ${\rm Lip}(\rr^d)$ such that $P_\cdot \varphi \in L^1(0, \infty; \CC_b^2(\rr^d))$, Eq. \eqref{ste} admits a solution $f_\varphi \in \CC_b^2(\rr^d)$ given by
	\begin{align} \label{fp}
	f_\varphi=-\int_0^{+\infty} [P_t \varphi -\pi(\varphi)]\dif t.
	\end{align}  
  \end{thm}

  \begin{proof}
Let $T>0$.
Applying $P_t$ to both sides of Eq. \eqref{ste}, integrating over $t\in[0,T]$, and using the definition \eqref{pt} and the identity \eqref{af}, we obtain
	\begin{align}
		\begin{aligned}
	\int_0^TP_t[\varphi(x)-\pi(\varphi)]\dif t
	& =\int_0^TP_t\mathcal Af_\varphi(x)\dif t
			=\int_0^T\ee [\mathcal Af_\varphi(X_t^x)]\dif t \\
		&	=\ee [f_\varphi(X_T^x)-f_\varphi(X_0^x)]
		=P_T f_\varphi(x)-f_\varphi(x).
		\end{aligned}
	\end{align} 
Letting $T\to+\infty$ and using the mixing property \eqref{mix}, we derive
	\begin{align*}
		f_\varphi(x)=-\int_0^{+\infty}P_t[\varphi(x)-\pi(\varphi)]\dif t+\pi(f_\varphi),
	\end{align*}
which shows that $f_\varphi$ given in \eqref{fp} solves the Stein equation \eqref{ste}.
\end{proof}

  \begin{rk}
  The solution to the Stein equation \eqref{ste} is not unique even up to a constant. 
  Indeed, in the case $d=m=1$, any function 
  $f:=f_\varphi+f_{b, \sigma}$ with $f_{b, \sigma}'(x)=c_0\exp(\int_0^x -\frac{2 b(y)}{\sigma(y)^2} \dif y)$ for any constant $c_0$ is also a solution of Eq. \eqref{ste}.
\end{rk}
 
In the rest of this paper, we always take $f_\varphi$ as defined in \eqref{fp}.

\subsection{Ergodic Estimate of One-step Approximations}

Let $\tau \in (0, 1)$ be a fixed step-size, $t_k:=\tau k$, and $\delta_k W:=W(t_{k+1})-W(t_k)$ for $k \in \nn$.
We consider the following general one-step approximation of Eq. \eqref{sde}:
\begin{align} \label{y}
		Y_{k+1}=Y_k+ B_{b, \sigma}(\tau, Y_k, Y_{k+1}, \delta_k W),
	\end{align}
	where the explicit or implicit function $B_{b, \sigma}$ depends on the coefficients $b$ and $\sigma$, the step-size $\tau$, the previous value $Y_k$, the current value $Y_{k+1}$, and the Brownian increment $\delta_k W$ over the subinterval $[t_k, t_{k+1}]$, $k \in \nn$.
	
	Denote by $Y_k^x$ the $k$-th iterative starting from $Y_0=x$.
	We assume that the discrete process $(Y_k)_{k\in\mathbb N}$ is a homogeneous Markov chain with a unique invariant measure $\pi_\tau$. 
	 
Assume that $(Y_k)_{k \in \nn}$ defined in \eqref{y} has a continuous interpolation $(\hat Y_t)_{t \ge 0}$ which is an It\^o process and has the following form on the subinterval $[0, \tau]$:
\begin{align} \label{yhat}
		\hat Y_t=\hat Y_0+\int_0^t \hat b_\tau(\tilde Y_0)\dif s+\int_0^t \hat \sigma_\tau(\tilde Y_0)\dif B_s, \quad t \in [0, \tau],
	\end{align} 
	with $\FFF_0$-measurable random variables $\hat Y_0$ and $\tilde Y_0$ depending on the data of Eq. \eqref{sde}.
	More precisely, we assume there exists two functions $g_\tau, \tilde g_\tau: \rr^d \to \rr^d$ such that $\hat Y_0=g_\tau(Y_0)$ and $\tilde Y_0=\tilde g_\tau(Y_0)$.
	
Denote by	$\hat Y_\tau^x$ the solution $\hat Y_\tau$ of \eqref{yhat} with initial datum $\hat Y_0=x$, and define a linear operator
\begin{align} \label{Atau}
\mathcal A_\tau f(x):=\ee [f (\hat Y_\tau^x)-f(x)],
\quad x \in \rr^d, ~ f \in \CC_b^2(\rr^d).
\end{align}
In particular, when $\hat Y_0=Y_0$ and $\hat Y_\tau=Y_1$, then $\mathcal A_\tau$ is the generator of the Markov chain $(Y_k)_{k\in\mathbb N}$.  
	Similarly to \eqref{af}, we have 
	\begin{align*}
				\mathcal A_\tau f(x)
				=&\ee \int_0^\tau \Big[\nabla f(\hat Y_s^x) \hat b_\tau(\tilde g_\tau(x))
				+\frac{1}{2}  \sum_{j=1}^m \nabla^2 f(\hat Y_s^x) (\hat \sigma_{j,\tau}(\tilde Y_0), \hat \sigma_{j,\tau}(\tilde g_\tau(x)))\Big] \dif s.
			\end{align*}   
Comparing with \eqref{df-a} of $\mathcal A f(x)$, we have the following representation of $\mathcal A_\tau f(x)$:
		\begin{align} \label{atau-a}
		\mathcal A_\tau f(y)=\tau \mathcal Af(x)
		+ \sum_{i=1}^6 \ee \mathcal R_i(x, y),  
		\end{align} 
		where 
\begin{align} \label{R}
\begin{split}
\mathcal R_1(x, y) & :=\int_0^\tau (\nabla f(\hat Y_s^y)-\nabla f(y)) \hat b_\tau(\tilde g_\tau(x))\dif s, \\
\mathcal R_2(x, y) & :=\frac{1}{2} \sum_{j=1}^m \int_0^\tau [\nabla^2 f(\hat Y_s^y)-\nabla^2 f(y)] (\hat \sigma_{j, \tau}(\tilde g_\tau(x)), \hat \sigma_{j, \tau}(\tilde g_\tau(x))) \dif s,  \\
\mathcal R_3(x, y) & := (\nabla f(y)-\nabla f(x)) \hat b_\tau(\tilde g_\tau(x)) \tau,   \\
\mathcal R_4(x, y) & :=\frac{1}{2} \sum_{j=1}^m [\nabla^2 f(y)-\nabla^2 f(x)] (\hat \sigma_{j, \tau}(\tilde g_\tau(x)), \hat \sigma_{j, \tau}(\tilde g_\tau(x)))  \tau,  \\
\mathcal R_5(x, y) & := \nabla f(x) [\hat b_\tau(\tilde g_\tau(x))-b(x)] \tau,    \\
\mathcal R_6(x, y) & :=\frac{1}{2} \sum_{j=1}^m [\nabla^2 f(x) (\hat \sigma_{j, \tau}(\tilde g_\tau(x)), \hat \sigma_{j, \tau}(\tilde g_\tau(x))) - \nabla^2 f(x) (\sigma_{j, \tau}(x), \sigma_{j, \tau}(x))] \tau.   
\end{split}
\end{align}
for $x, y \in \rr^d$. 
	 
We now state the main result on the error representation between the exact invariant measure $\pi$ of Eq. \eqref{sde} and the numerical invariant measure $\pi_\tau$ of the associated one-step approximation \eqref{y}, along with the corresponding error estimate for the ergodic limit.

	\begin{thm} \label{tm-err}
For any $\varphi \in \BB_b(\rr^d)$ or ${\rm Lip}(\rr^d)$ such that $P_\cdot \varphi \in L^1(0, \infty; \CC_b^2(\rr^d))$ and $X_0 \sim \pi_\tau$, there holds that 
		\begin{align} \label{err}
		|\pi_\tau(\varphi)-\pi(\varphi)|
=\tau^{-1} \Big|\ee ^{\pi_\tau}\mathcal A_\tau f_\varphi(g_\tau(X_0))
- \sum_{i=1}^6 \ee ^{\pi_\tau} [\ee \mathcal R_i(x, g_\tau(x))|_{x=X_0}] \Big|.
		\end{align}
Assume furthermore that the one-step approximation \eqref{y} is $V$-geometrically ergodic (polynomially mixing), i.e., there exist positive constants $C$ and $\rho \in (0, 1)$ such that $|\mathbb E\varphi(Y_k^x)-\pi_\tau(\varphi)| \le CV(x)\rho^k$.
Then
\begin{align} \label{lim}
& \Big|\frac{1}{N}\sum_{k=0}^{N-1}\mathbb E\varphi(Y_k^x)-\pi(\varphi) \Big| \nonumber  \\
& \le \frac{CV(x)}{(1-\rho)N}+ \tau^{-1} \Big|\ee ^{\pi_\tau}\mathcal A_\tau f_\varphi(g_\tau(X_0)) - \sum_{i=1}^6 \ee ^{\pi_\tau} [\ee \mathcal R_i(x, g_\tau(x))|_{x=X_0}] \Big|.		
		\end{align}
	\end{thm}

	\begin{proof}
Let $X_0 \sim \pi_\tau$ and $\varphi$ as stated.
From the Stein equation \eqref{ste} and the representation \eqref{atau-a}, we have 
\begin{align*}
 |\pi_\tau(\varphi)-\pi(\varphi)|
=& |\ee^{\pi_\tau} \varphi(X_0)-\pi(\varphi)| \\
=& |\ee^{\pi_\tau}\mathcal Af_\varphi(X_0)| \\
=& \tau^{-1} \Big|\ee ^{\pi_\tau}\mathcal A_\tau f_\varphi(g_\tau(X_0))
- \sum_{i=1}^6  \ee ^{\pi_\tau}  [\ee \mathcal R_i(x, g_\tau(x))|_{x=X_0}] \Big|,
\end{align*}
which proves \eqref{err}.

By triangle inequality and the $V$-geometrical ergodicity of $\{Y_k\}$, we obtain 
\begin{align*}
\Big|\frac{1}{N}\sum_{k=0}^{N-1}\mathbb E\varphi(Y_k^x)-\pi_\tau(\varphi) \Big|
& \le \Big|\frac{1}{N}\sum_{k=0}^{N-1}\mathbb E\varphi(Y_k^x)-\pi_\tau(\varphi) \Big|
+ |\pi_\tau(\varphi)-\pi(\varphi)| \\
& = \frac{1}{N}\sum_{k=0}^{N-1} |\mathbb E\varphi(Y_k^x)-\pi_\tau(\varphi) |
+ |\pi_\tau(\varphi)-\pi(\varphi)|.
\end{align*}
Inequality \eqref{lim} then follows from the above estimate and the representation \eqref{err}.
	\end{proof}

\section{Applications to Superlinear SODEs}
\label{sec3}

This section applies the results of Theorem \ref{tm-err} to a family of one-step approximations for superlinear SODEs.

Although the same methodology can handle a broad class of examples, we focus on several recently studied one-step numerical schemes, including the tamed Euler, projected Euler, and backward Euler methods, in Example \ref{ex-sde}.

\subsection{Assumptions and Preliminary Estimates}

We impose the following frequently used dissipativity and polynomial growth conditions.

\begin{ap} \label{A1}
There exist positive constants $L_1,L_2, L_3$, $p^{\star} \ge 2$, and $\gamma >1$ such that for any $x,y\in\rr^d$,
     \begin{align}
       &  \< 
         x-y, b(x)-b(y)
         \> 
         +
         \frac{2p^{\star}-1}{2}
         \|\sigma(x)
         -
         \sigma(y)\|^2
         \le  -L_1 |x-y|^2,\label{mon}\\
   & \< x, b(x)\>
    +
    \frac{p^{\star}(2p^{\star}-1)}{2}
    \|\sigma(x)\|_{HS}^2
    \le 
   L_2 -L_3|x|^{\gamma+1}.\label{coe}
\end{align}
\end{ap}

\begin{ap} \label{A2}
The functions $b,\sigma$ have continuous derivatives up to order four, and there exists $C>0$ such that for any $k=1,2,3,4$ and $x, v_k \in \rr^d$, 
    \begin{align*}
            |
            \nabla^k b(x)
            (v_1,\cdots, v_k)
            | 
            & \le  
            C
           \big [
            \1_{\gamma \le  k}
            +
            \1_{\gamma >k}
            (1+|x|^{\gamma-k})
           \big ]|v_1| 
            \cdots |v_k|, \\
            |
            \nabla^k \sigma_j(x)
            (v_1, \cdots, v_k)
            |^2 
            &  \le  
            C
            [
            \1_{\gamma \le  2k-1}
            +
            \1_{\gamma >2k-1}
            (1+|x|^{\gamma-(2k-1)})
            ]|v_1|^2\cdots|v_k|^2.
    \end{align*} 
\end{ap}

\begin{rk}
Under Assumption \ref{A1}, Eq. \eqref{sde} has a unique strong solution, which is a Markov process exponentially ergodic to a unique invariant measure $\pi$ \cite[Theorem 3.1.1, Proposition 4.3.5, and Theorem 4.3.9]{LR15}.
Moreover, if $X_0 \in L^{2p^*}$, then  
\begin{align*}
\sup_{t \ge 0} \ee |X_t|^{2p^*} \le C (1+ \ee |X_0|^{2p^*}),  
\end{align*}
provided $X_0 \in L^{2p^*}$.
Under the additional Assumption \ref{A2}, it is clear, for all $T>0$ and $j=\{1,2,\cdots,m\}$, that $\ee \int_0^T |\sigma_j(X_t)|^2 \dd t  \le C T (1+ \ee |X_0|^{\gamma+1}) <\infty$ so that $\int_0^\cdot \sigma(X_s) \dd W(s)$ is a martingale; see, e.g., \cite{Liu22}. 
\end{rk}

\subsection{Estimates of Stein Equation}

To apply Theorems \ref{tm-ste} and \ref{tm-err} in the present settings, we need to show that $P_\cdot \varphi \in L^1(0, \infty; \CC_b^2(\rr^d))$ for a given test function $\varphi$ and to derive necessary derivative estimates for the Stein equation solution $f_\varphi$. 

We have the following estimates on the directional derivatives $\nabla^i X_t^\cdot$ with $i=1, 2, 3, 4$ and $t \ge 0$; see \cite[Lemma 3.3]{LWWZ25} for a weak result in the case of weak dissipativity.

\begin{prop} \label{prop-DX}
Let Assumptions \ref{A1} and \ref{A2} hold and $q \in (1, p^{\star})$.  
There exist positive constants $\lambda$, $\gamma_1^*$, and $C$ such that  
\begin{align}   \label{DX-est}
\sum_{i=1}^4 \|\nabla^i X_t^x\|_{L^{2q}(\Omega; \LL((\rr^d)^{\otimes i}; \rr))}
    & \le  C e^{-\lambda t} (1+ |x|^{\gamma_1^*}), \quad x \in \rr^d, ~ t \ge 0. 
\end{align}
Consequently, $P_\cdot \varphi \in L^1(0, \infty; \CC_b^2(\rr^d))$ for any $\varphi \in \CC_b^2(\rr^d)$ and there exist positive constant $C$ and $\gamma_1$ such that for any $\varphi \in \CC_b^i(\rr^d)$,  
\begin{align}   \label{est-fp}
\|\nabla^i f_\varphi^x\|_{\LL((\rr^d)^{\otimes i}; \rr)}
    & \le   C (1+|x|^{\gamma_1})  \|\nabla^i \varphi\|, \quad i=1,2,3,4. 
\end{align}
\end{prop}

\begin{proof} 
Let $q \in (1,p^{\star})$, $t \ge 0$, and $x, v_1, v_2, v_3, v_4 \in \rr^d$.
In what follows, we denote
\begin{align*}
    \eta^{v_1}(t, x)
    := \nabla X_t^{x} v_1, \quad 
    & \xi^{v_1, v_2}(t, x)
    := \nabla^2 X_t^{x}(v_1, v_2), \\ 
    \zeta^{v_1, v_2,v_3}(t, x)
    :=  \nabla^3 X_t^{x}(v_1, v_2,v_3), \quad 
    & \varpi^{v_1, v_2,v_3, v_4}(t, x)
    :=  \nabla^4 X_t^{x}(v_1, v_2,v_3, v_4).  
\end{align*} 

It follows from \eqref{mon} that
\begin{align}
\label{D-coup}
    \< \nabla b(x) y, y\> +
    \frac{2p^{\star}-1}{2}
    \|D\sigma(x)y\|^2\le  -L_1|y|^2,
    \quad x, y \in \rr^d.
\end{align}
For $\eta^{v_1}(t,x)$, we have $\eta^{v_1}(0,x)=v_1$ and for $t>0$, 
\begin{align*}
\dd \eta^{v_1}(t,x)
    =
    \nabla b(X_t^x) \eta^{v_1}(t, x) 
   \dd t
    +
    \sum_{j=1}^m
    \nabla \sigma_j(X_t^x) \eta^{v_1}(t, x) 
   \dd W_j(t).
\end{align*}
Define a stopping time 
\begin{align*}
\tau_n:
= \inf \{ s \geq 0:|\eta^{v_1}(s,x)|>n\}.
\end{align*}
Using It\^o formula and the inequality \eqref{D-coup} implies
\begin{align*}
        & \frac1{2q} [\ee | \eta^{v_1}(t \wedge \tau_n, x)|^{2q}- |v_1|^{2q}] \nonumber\\
        & =\ee \int_0^{t \wedge \tau_n}
 |\eta^{v_1}( x)|^{2(q-1)} \Big[ \< \eta^{v_1}(x), \nabla b(X^x) \eta^{v_1}(x) \> \nonumber\\
& \qquad + \frac{2q-1}{2} \sum_{j=1}^m |\nabla \sigma_j(X^x) \eta^{v_1}(x)  |^2 \Big]
       \dd s \nonumber\\
        & \le  -L_1  \ee \int_0^{t \wedge \tau_n}
        |\eta^{v_1}(x)|^{2q}
       \dd s.
\end{align*}
We omit the integration variable here and after when there is an integration to lighten the notation. 
Using Gronwall inequality and Fatou lemma, we get 
\begin{align}
        \ee|\nabla X_t^x v_1 |^{2q}
       & \le   e^{-2q\lambda t}|v_1|^{2q}. \label{D1X-est} 
\end{align}

Similarly, $ \xi^{v_1, v_2}(t, x)$ satisfies $ \xi^{v_1, v_2}(0, x)=0 $ and
\begin{align*}
  \dd \xi^{v_1, v_2}(t, x)
    &= [\nabla b(X_t^x)   \xi^{v_1, v_2}(t, x) +  \nabla^2 b(X_t^x)
        (\eta^{v_1}(t, x), \eta^{v_2}(t, x))]
   \dd t  \\
    &\quad +
    \sum_{j=1}^m [\nabla \sigma_j(X_t^x) \xi^{v_1, v_2}(t, x)  +
        \nabla^2 \sigma_j(X_t^x) (\eta^{v_1}(t, x),  \eta^{v_2}(t, x))] 
   \dd W_j(t),
      \end{align*}
   so that (introducing an analogous stopping time if necessary)
   \begin{align*} 
        &\frac{1}{2q}\ee|\xi^{v_1,v_2}(t,x)|^{2q} \nonumber\\
      &  = \ee \int_0^t 
 |\xi^{v_1,v_2}(x)|^{2(q-1)} \Big[ \< \xi^{v_1,v_2}(x), \nabla b(X^x) \xi^{v_1,v_2}(x)
 +  \nabla^2 b(X^x) (\eta^{v_1}(x), \eta^{v_2}(x)  ) \> \nonumber\\
& \qquad + \frac{2q-1}{2} \sum_{j=1}^m |\nabla \sigma_j(X^x) \xi^{v_1, v_2}(x)  +
        \nabla^2 \sigma_j(X^x) (\eta^{v_1}(x),  \eta^{v_2}(x))|^2 \Big]
       \dd s.
\end{align*} 
By Young inequality and \eqref{D-coup}, for an arbitrarily small $\epsilon>0$ such that $\frac{2q-1}{2} (1+\epsilon) \le \frac{2p^{\star}-1}{2}$, there exists $\delta=\delta(\epsilon) \in (0, L_1)$ and $C_{\epsilon, \delta}>0$ such that   
\begin{align*}
        &\frac{1}{2q}\ee|\xi^{v_1,v_2}(t,x)|^{2q}  \\
       & \le \ee \int_0^t |\xi^{v_1,v_2}(x)|^{2(q-1)} \big[\epsilon|\xi^{v_1,v_2}(x)|^2 
        + (2\epsilon)^{-1}|\nabla^2 b(X^x)(\eta^{v_1}(x),\eta^{v_2}(x))|^2\big]\dd s \nonumber\\
        & \quad +\ee\int_0^t |\xi^{v_1,v_2}(x)|^{2(q-1)} [\<\xi^{v_1,v_2}(x),\nabla b(X^x)\xi^{v_1,v_2}(x)\> \nonumber\\
 & \qquad \qquad \qquad \qquad \qquad \qquad + \frac{2q-1}{2} (1+\epsilon) \sum_{j=1}^m |\nabla\sigma_j(X^x)\xi^{v_1,v_2}(x)|^2] \dd s \nonumber\\
        & \quad + \frac{2q-1}{2} (1+\epsilon^{-1}) \sum_{j=1}^m  \ee\int_0^t |\xi^{v_1,v_2}(x)|^{2(q-1)} |\nabla^2\sigma_j(X^x)(\eta^{v_1}(x),\eta^{v_2}(x))|^2 \dd s \nonumber\\
        & \le -(L_1-\delta)\int_0^t\ee |\xi^{v_1,v_2}(x)|^{2q}\dd s + C_{\epsilon,\delta} \underbrace{\ee\int_0^{t }|\nabla^2 b(X^x)  (\eta^{v_1}(x), \eta^{v_2}(x))|^{2q} \dd s}_{I_{11}} \nonumber\\
        & \quad + C_{\epsilon,\delta} \underbrace{\sum_{j=1}^m\ee\int_0^{t }|
        	\nabla^2 \sigma_j(X^x) (\eta^{v_1}(x), \eta^{v_2}(x))
        	|^{2q} \dd s}_{I_{12}}.  
\end{align*}

It remains to estimate $I_{11}$ and $I_{12}$.
We employ Assumption \ref{A2},  H\"older inequality, and the estimate \eqref{D1X-est} to arrive at 
\begin{align*}
	I_{11}\le & C \ee \int_0^t\big[\1_{\gamma\in(1,2]}+\1_{\gamma>2}(1+|X^x|^{\gamma-2})\big]^{2q} |\eta^{v_1}(x)|^{2q}|\eta^{v_2}(x)|^{2q} \dd s\nonumber\\
	\le & C\big[\1_{\gamma\in(1,2]}+\1_{\gamma>2}(1+|x|^{2q(\gamma-2)})\big]|v_1|^{2q}|v_2|^{2q}.
\end{align*}
In the same manner,
\begin{align*}
       I_{12}\le C\big[\1_{\gamma\in(1,3]}+\1_{\gamma>3}(1+|x|^{q(\gamma-3)})\big]|v_1|^{2q}|v_2|^{2q}.
\end{align*}
Combining the above two estimates, we have
\begin{align*}
        & \frac1{2q} \ee |\xi^{v_1, v_2} (t, x)|^{2q} + (L_1-\delta) \ee\int_0^{t }|\xi^{v_1, v_2}(x)|^{2q}\dd s \\
        & \le  C [ \1_{\gamma \in(1,2]} + \1_{\gamma \in [2,3]} (1+|x|^{2q(\gamma-2)})
        + \1_{\gamma >3}(1+|x|^{q(\gamma-3)})]|v_1|^{2q} |v_2|^{2q}.
\end{align*}
and in combination with Gronwall inequality, we derive
\begin{align}
\ee|\nabla^2 X_t^x(v_1, v_2) |^{2q}
        & \le C  e^{-2q\lambda t} 
        [ \1_{\gamma\in(1,2]} + \1_{\gamma >2}
       ( 1+  |x|^{2q(\gamma-2)} ) ]
        |v_1|^{2q} |v_2|^{2q}. \label{D2X-est}
\end{align}

For the third term $\zeta^{v_1,v_2,v_3}(t,{x})$, we get $\zeta^{v_1, v_2, v_3}(0, x)=0$ and 
\begin{align*}
        & \zeta^{v_1, v_2, v_3}(t, x)\\
        & =  
        (
        \nabla b(X_t^x) \zeta^{v_1, v_2, v_3}(t, x)
        +
        \nabla^2 b(X_t^x)
       (
        \eta^{v_1}(t, x), 
        \xi^{v_2, v_3}(t, x)
        )\\
        & \quad +
        \nabla^2b(X_t^x)
       (
        \xi^{v_1, v_3}(t, x), \eta^{v_2}(t, x)
       ) 
        +
        \nabla^2 b(X_t^x)
        (
        \xi^{v_1, v_2}(t, x), \eta^{v_3}(t, x)
        )\\
        & \quad +
        \nabla^3 b(X_t^x)
        (
        \eta^{v_1}(t, x), 
        \eta^{v_2}(t, x), 
        \eta^{v_3}(t, x)
        )
        ) 
       \dd t \\
        & \quad +
        \sum_{j=1}^m
        (
        \nabla \sigma_j
      (X_t^x) 
        \zeta^{v_1, v_2, v_3}(t, x)
        +
        \nabla^2 \sigma_j(X_t^x)
        (
        \eta^{v_1}(t, x), 
        \xi^{v_2, v_3}(t, x)
        )\\
        & \quad + 
         \nabla^2 \sigma_j(X_t^x)
        (
         \xi^{v_1, v_3}(t, x), \eta^{v_2}(t, x)
         )
         +
         \nabla^2 \sigma_j
         (X_t^x)
         (
         \xi^{v_2, v_3}(t, x), \eta^{v_1}(t, x)
         )\\
         & \quad +
         \nabla^3 \sigma_j
         (X_t^x)
         (
         \eta^{v_1}(t, x), 
         \eta^{v_2}(t, x), 
         \eta^{v_3}(t, x)
         )
         ) 
        \dd W_j(t)
         \\
         & = :
         (
         \nabla b
         (X_t^x) \zeta^{v_1, v_2, v_3}(t, x)
         +
         H(X_t^x))\dd t\\
         & \quad +
         \sum_{j=1}^m
         (
         \nabla \sigma_j(X_t^x) \zeta^{v_1, v_2, v_3}(t, x)
         +
         G_j(X_t^x)
         ) 
        \dd W_j(t).
\end{align*} 
As in the same arguments to deal with the term $\xi^{v_1,v_2}(t,x)$, we have
\begin{align*}
        &\quad \frac{1}{2q}\ee|\zeta^{v_1, v_2, v_3}(t , x)|^{2q}+ (L_1-\delta) \ee   \int_0^{t }|\zeta^{v_1, v_2, v_3}( x)|^{2q} \dd s \nonumber \\
        & \le C \underbrace{\ee   \int_0^{t }|H(X^x)|^{2q} \dd s}_{I_{21}}
    +C\underbrace{\sum_{j=1}^m \ee \int_0^t|G_j(X^x)|^{2q}\dd s}_{I_{22}}.
\end{align*}
We split the estimation of the term $I_{21}$ into the following four parts:
\begin{align*}
        I_{21} & \le 
    C    \int_0^{t } \ee |
        \nabla^2 b(X^x)
        (
        \eta^{v_1}(x), 
        \xi^{v_2, v_3}(x)
        )
      |^{2q}\dd s
      \\
        & \quad +
    C    \int_0^{t } \ee |
        \nabla^2 b(X^x)
        (
        \xi^{v_1, v_3}(x), \eta^{v_2}(x)
        )
      |^{2q} \dd s
        \\
        & \quad +
     C    \int_0^{t } \ee
      |
        \nabla^2 
        b(X^x)
        (\xi^{v_1, v_2}(x), \eta^{v_3}(x)
        )
        |^{2q} \dd s
        \\
        & \quad +C    \int_0^{t } \ee
       |
        \nabla^3 b(X^x)
        (
        \eta^{v_1}(x), 
        \eta^{v_2}(x), 
        \eta^{v_3}(x)
        )
        |^{2q} \dd s  \\
        &=:  I_{211}+ I_{212}+ I_{213}+ I_{214}.
\end{align*}
We mention that the analysis of $ I_{211}, I_{212}, I_{213}$ is similar and we take $ I_{211}$ as an example.
By Assumption \ref{A2}, H\"older inequality, and the estimates \eqref{D1X-est} and \eqref{D2X-est}, we have
\begin{align*}
     I_{211} \le  C (1 + |x|^{4q(\gamma-2)})|v_1|^{2q}| v_2|^{2q}.
\end{align*}
For $I_{214}$, we have
\begin{align*}
       I_{214}
         \le  
        C [
        \1_{\gamma \in(1,3]}
        +
        \1_{\gamma >3}
        (1+|x|^{2q(\gamma-3)})
        ]
        |v_1|^{2q}  |v_2|^{2q} |v_3|^{2q}.
\end{align*}
%
%
Hence, we get
\begin{align*}
      I_{21}
        & \le 
        C  [
        \1_{1 < \gamma \le 2}
        +
        \1_{\gamma >2}
        (1+|x|^{4q(\gamma-2)})
        ]
        |v_1|^{2q}
        |v_2|^{2q}
        |v_3|^{2q}.
\end{align*}
Similarly, it is not hard to obtain
\begin{align*}
       I_{22}& \le 
        C [
        \1_{\gamma \in(1,2]}
        +
        \1_{\gamma >2}
        (1+|x|^{4q(\gamma-2)})
        ] |v_1|^{2q}|v_2|^{2q}|v_3|^{2q}.
\end{align*}
Plugging these estimates together yields
\begin{align*}
       & \frac{1}{2q}\ee [ |\zeta^{v_1, v_2, v_3}
        (t, x)|^{2q}
        + (L_1-\delta) \ee
        \int_0^{t }
        |\zeta^{v_1, v_2, v_3}(x)|^{2q}
       \dd s \nonumber\\
        & \le 
        C [ \1_{1 < \gamma \le 2}
        +
        \1_{\gamma >2}
        (1+|x|^{4q(\gamma-2)})
        ]
        |v_1|^{2q}
        |v_2|^{2q}
        |v_3|^{2q}.
\end{align*}
By Gronwall inequality, we conclude 
\begin{align} \label{D3X-est}
  \ee|\nabla^3 X_t^x(v_1, v_2,v_3)|^{2q}
    & \le  C e^{-2q\lambda t} 
    [\1_{1 < \gamma \le 2}+\1_{\gamma >2}
   ( 1+  |x|^{4q(\gamma-2)} ) ] |v_1|^{2q} |v_2|^{2q}  |v_3|^{2q}.
\end{align} 

For the last term, we get $\varpi^{v_1,v_2,v_3,v_4}(0,x)=0$ and 

\begin{align*}
	\varpi^{v_1,v_2,v_3,v_4}(t,x)=&(\nabla b(X_t^x)\varpi^{v_1,v_2,v_3,v_4}(t,x)+J_{31})\dd t\\
	&+\sum_{j=1}^m(\nabla \sigma_j(X_t^x)\varpi^{v_1,v_2,v_3,v_4}(t,x)+J^j_{32})\dd W_j(t).
\end{align*}
where
\begin{align*}
	J_{31}=&\sum_{i=1}^{14} J_{31i}\\
	=&\nabla^2b(X_t^x)(\eta^{v_1}(t,x),\zeta^{v_2,v_3,v_4}(t,x))\\
	&+\nabla^2b(X_t^x)(\eta^{v_2}(t,x),\zeta^{v_1,v_3,v_4}(t,x))\\
	&+\nabla^2b(X_t^x)(\eta^{v_3}(t,x),\zeta^{v_1,v_2,v_4}(t,x))\\
	&+\nabla^2b(X_t^x)(\eta^{v_4}(t,x),\zeta^{v_1,v_2,v_3}(t,x))\\
	&+\nabla^2b(X_t^x)(\xi^{v_1,v_2}(t,x),\xi^{v_3,v_4}(t,x))\\
	&+\nabla^2b(X_t^x)(\xi^{v_1,v_3}(t,x),\xi^{v_2,v_4}(t,x))\\
	&+\nabla^2b(X_t^x)(\xi^{v_1,v_4}(t,x),\xi^{v_2,v_3}(t,x))\\
	&+\nabla^3 b(X_t^x)(\xi^{v_1,v_2}(t,x),\eta^{v_3}(t,x),\eta^{v_4}(t,x))\\
	&+\nabla^3 b(X_t^x)(\xi^{v_1,v_3}(t,x),\eta^{v_2}(t,x),\eta^{v_4}(t,x))\\
	&+\nabla^3 b(X_t^x)(\xi^{v_1,v_4}(t,x),\eta^{v_2}(t,x),\eta^{v_3}(t,x))\\
	&+\nabla^3 b(X_t^x)(\xi^{v_2,v_3}(t,x),\eta^{v_1}(t,x),\eta^{v_4}(t,x))\\
	&+\nabla^3 b(X_t^x)(\xi^{v_2,v_4}(t,x),\eta^{v_1}(t,x),\eta^{v_3}(t,x))\\
	&+\nabla^3 b(X_t^x)(\xi^{v_3,v_4}(t,x),\eta^{v_1}(t,x),\eta^{v_2}(t,x))\\
	&+\nabla^4 b(X_t^x)(\eta^{v_1}(t,x),\eta^{v_2}(t,x),\eta^{v_3}(t,x),\eta^{v_4}(t,x)),  \\
	J^j_{32}=&\sum_{i=1}^{14} J^j_{32i}\\
	=&\nabla^2\sigma_j(X_t^x)(\eta^{v_1}(t,x),\zeta^{v_2,v_3,v_4}(t,x))\\
	&+\nabla^2\sigma_j(X_t^x)(\eta^{v_2}(t,x),\zeta^{v_1,v_3,v_4}(t,x))\\
	&+\nabla^2\sigma_j(X_t^x)(\eta^{v_3}(t,x),\zeta^{v_1,v_2,v_4}(t,x))\\
	&+\nabla^2\sigma_j(X_t^x)(\eta^{v_4}(t,x),\zeta^{v_1,v_2,v_3}(t,x))\\
	&+\nabla^2\sigma_j(X_t^x)(\xi^{v_1,v_2}(t,x),\xi^{v_3,v_4}(t,x))\\
	&+\nabla^2\sigma_j(X_t^x)(\xi^{v_1,v_3}(t,x),\xi^{v_2,v_4}(t,x))\\
	&+\nabla^2\sigma_j(X_t^x)(\xi^{v_1,v_4}(t,x),\xi^{v_2,v_3}(t,x))\\
	&+\nabla^3 \sigma_j(X_t^x)(\xi^{v_1,v_2}(t,x),\eta^{v_3}(x,t),\eta^{v_4}(t,x))\\
	&+\nabla^3 \sigma_j(X_t^x)(\xi^{v_1,v_3}(t,x),\eta^{v_2}(t,x),\eta^{v_4}(t,x))\\
	&+\nabla^3 \sigma_j(X_t^x)(\xi^{v_1,v_4}(t,x),\eta^{v_2}(t,x),\eta^{v_3}(t,x))\\
	&+\nabla^3 \sigma_j(X_t^x)(\xi^{v_2,v_3}(t,x),\eta^{v_1}(t,x),\eta^{v_4}(t,x))\\
	&+\nabla^3 \sigma_j(X_t^x)(\xi^{v_2,v_4}(t,x),\eta^{v_1}(t,x),\eta^{v_3}(t,x))\\
	&+\nabla^3 \sigma_j(X_t^x)(\xi^{v_3,v_4}(t,x),\eta^{v_1}(t,x),\eta^{v_2}(t,x))\\
	&+\nabla^4 \sigma_j(X_t^x)(\eta^{v_1}(t,x),\eta^{v_2}(t,x),\eta^{v_3}(t,x),\eta^{v_4}(t,x)).
\end{align*}
Then we have 
\begin{align*}
	&\frac{1}{2q}\ee |\varpi^{v_1,v_2,v_3,v_4}(t,x)|^{2q}+(L_1-\delta)\ee\int_0^t  |\varpi^{v_1,v_2,v_3,v_4}(x)|^{2q}\dd s\\
	\le &C\ee\int_0^t|J_{31}|^{2q}\dd s+C\sum_{j=1}^m \ee\int_0^t |J^j_{32}|^{2q}\dd s\\
	\le &C\sum_{i=1}^{14}\ee\int_0^t|J_{31i}|^{2q}\dd s+C \sum_{i=1}^{14}\sum_{j=1}^m\ee\int_0^t |J^j_{32i}|^{2q}\dd s.
\end{align*}
For $\sum_{i=1}^4 \ee\int_0^t |J_{31i}|^{2q}\dd s$, we employ Assumption \eqref{A2}, H\"older inequality, the estimates \eqref{D1X-est}, \eqref{D2X-est}, and \eqref{D3X-est} to get
\begin{align*}
\sum_{i=1}^{14} \ee\int_0^t |J_{31i}|^{2q}\dd s
	\le C[\1_{\gamma\in(1,2]}+\1_{\gamma\geq 2}(1+|x|^{6q(\gamma-2)})]|v_1|^{2q}|v_2|^{2q}|v_3|^{2q}|v_4|^{2q}.
\end{align*}
Similarly, we have
\begin{align*}
	\sum_{i=1}^{14} \sum_{j=1}^{m}\ee\int_0^t |J^j_{31i}|^{2q}\dd s
	\le C[\1_{\gamma\in(1,2]}+\1_{\gamma\geq 2}(1+|x|^{6q(\gamma-2)})]|v_1|^{2q}|v_2|^{2q}|v_3|^{2q}|v_4|^{2q}.
\end{align*}
Hence
\begin{align*}
	&\frac{1}{2q}\ee |\varpi^{v_1,v_2,v_3,v_4}(t,x)|^{2q}+(L_1-\delta)\ee\int_0^t  |\varpi^{v_1,v_2,v_3,v_4}(x)|^{2q}\dd s\\
	& \le C[\1_{\gamma\in(1,2]}+\1_{\gamma\geq 2}(1+|x|^{6q(\gamma-2)})]|v_1|^{2q}|v_2|^{2q}|v_3|^{2q}|v_4|^{2q},
\end{align*}
and we conclude 
\begin{align} \label{D4X-est}
  & \ee|\nabla^4 X_t^x(v_1, v_2,v_3, v_4)|^{2q} \nonumber \\
    & \le C e^{-2q\lambda t} 
    [ \1_{\gamma\in(1,2]} + \1_{\gamma >2}
   ( 1+  |x|^{6q(\gamma-2)} ) ] |v_1|^{2q} |v_2|^{2q}  |v_3|^{2q} |v_4|^{2q}.
\end{align} 
Thus, we obtain the estimate \eqref{DX-est} from \eqref{D1X-est}, \eqref{D2X-est}, \eqref{D3X-est}, and \eqref{D4X-est}.

Finally, it follows from \eqref{fp}, \eqref{pt}, and \eqref{DX-est} with $i=1$ that 
	 \begin{align*}
\|\nabla f_\varphi(x)\|_{\LL(\rr^d)}
& \le   \|\nabla \varphi\|_1 \int_0^{+\infty} \ee \|\nabla X_t^x\|_{\LL(\rr^d)} \dif t 
\le  C_1 (1+|x|^{\gamma_1})  \|\nabla \varphi\|_1,
 \end{align*}
 which shows the estimate \eqref{est-fp} with $i=1$.
 The remaining estimates  \eqref{est-fp} with $i=2, 3, 4$ are analogously, and we omit the details.
\end{proof}

\subsection{Ergodic Estimates of Numerical Schemes}

We now apply Theorem \ref{tm-err} to specific implicit and explicit numerical schemes.

\begin{ex} \label{ex-sde}
(1) \textbf{MEM (Modified Euler Method):}
\begin{align} \label{MEM} 
Y_{n+1}
=
\PPP(Y_n)
+ b_\tau(\PPP(Y_n)) \tau
+ \sigma_\tau(\PPP(Y_n))  
\delta_n W, \quad n \in \nn,
\end{align}
where $\PPP:\rr^d \rightarrow \rr^d$ is a modification function, and $b_\tau(x)$ and $\sigma_{j,\tau}(x)$, $j \in \{1,2,\cdots,m\}$, are $\rr^d$-valued measurable functions.
This includes the following choices.

(a) \textbf{TEM (Tamed Euler Method):}  
    \begin{align} \label{tem}
\begin{split}
\PPP(x):=x, ~ 
b_\tau(x):=\frac{b(x)}{(1+\tau|x|^{4(\gamma-1)})^{1/4}}, \quad 
\sigma_\tau(x):=\frac{\sigma(x)}{(1+\tau|x|^{4(\gamma-1)})^{1/4}},
        \quad x \in \rr^d.
        \end{split}
    \end{align}
    
(b) \textbf{PEM (Projected Euler Method):}
\begin{align}\label{pem}
\begin{split}
         \PPP(x)
         &:=
         \left\{\begin{array}{ll}
         \min \{1, \tau^{-\frac{1}{2\gamma}}|x|^{-1}\} x, & x \neq 0, \\
         0, &  x=0,
         \end{array} 
         \right. ~
         b_\tau(x):= b(x), ~ 
         \sigma_\tau(x):=\sigma(x), \quad x \in \rr^d.
         \end{split}
\end{align} 

(2)
\textbf{BEM (Backward Euler Method):}
	\begin{align} \label{bem}
		Y_{n+1}=Y_n+b(Y_{n+1})\tau
		+ \sigma(Y_n)\delta_n W,
		\quad n \in \nn.
	\end{align} 
\end{ex}

\begin{rk} \label{rk}
(1) The discrete processes $(Y_k)_{k\in\mathbb N}$ given in Example \ref{ex-sde} are homogeneous Markov chains and possess unique invariant measures, all denoted by $\pi_\tau$.
They are $V$--geometrically ergodic with $V(x):=1+|x|^2$, $x \in \rr^d$; see  \cite{LW25} for the BEM \eqref{bem} and \cite{LWWZ25} for the TEM \eqref{tem} and the PEM \eqref{pem}.

(2) The following estimates hold \cite[Example 2.2]{LWWZ25}: there exist positive constants $C$, $ \alpha_1 \ge 1$, and $ \alpha_2 \ge 1$ such that for any $x \in \rr^d$ and $j \in \{1,2,\cdots,m\}$,
    \begin{align} \label{est-bem-}
    \begin{split} 
    |\PPP(x)| &\le  |x|, \\ 
            | b_\tau(\PPP(x)) | & \le C|b(x)|,\\
            | \sigma_{j,\tau}(\PPP(x))  | &\le C|\sigma_j(x)|,\\
            |b_\tau(\PPP(x))-b(\PPP(x))|+
          \sum_{j=1}^m |\sigma_{j,\tau}(\PPP(x))-\sigma_j(\PPP(x))|
            &\le C \tau (1+|x|^{\alpha_1}), \\ 
    |\PPP(x)-x| & \le  C \tau^2 |x|^{\alpha_2}.
    \end{split}
\end{align} 
For TEM \eqref{tem}, $\alpha_1=5\gamma-4$ and $\alpha_2>0$ arbitrary; for PEM \eqref{pem}, $\alpha_1=1$ and $\alpha_2=4\gamma+1$.
\end{rk}

	\begin{thm} \label{tm-sde}
	Let Assumptions \ref{A1} and \ref{A2} hold with $2 p^{\star} \ge \max\{5\gamma-4, 4\gamma+1\}$.  
There exists a positive constant $C$ such that 
		\begin{align}  
\sup_{\|\varphi\|_4 \le 1} |\pi_\tau(\varphi)-\pi(\varphi)|
& \le C \tau, \label{err-tau}   \\ 
 \Big|\frac{1}{N}\sum_{k=0}^{N-1}\mathbb E\varphi(Y_k^x)-\pi(\varphi) \Big|
& \le C(1+|x|^2) (\tau+N^{-1}). \label{err-N} 		
		\end{align}
\end{thm} 
	
	\begin{proof}
	Let $\|\varphi\|_4 \le 1$.
	
	(1) For the MEM \eqref{MEM}, the associated intepolation process $(\hat Y_t)_{t\in[0,\tau]}$ can be choosen as 
\begin{align}
\hat Y_t=\PPP(Y_0)
+ \int_0^t b_\tau(\PPP(Y_0)) \dif s
+ \int_0^t \sigma_\tau(\PPP(Y_0))  
\dif W(s), \quad t \in [0, \tau],
	\end{align}
	i.e., $g_\tau(x)=\tilde g_\tau(x)=\PPP(x)$, $x \in \rr^d$.
It is straightforward to verify that $\hat Y_\tau=Y_1$, and that $\hat Y$ is a continuous It\^o process.  

Let us first estimate the term $\ee^{\pi_\tau} \mathcal A_\tau f_\varphi(\PPP(Y_0))$ with $Y_0=X_0 \sim \pi_\tau$. 
As $\pi_\tau$ is an invariant measure of \eqref{MEM}, we have $\hat Y_\tau=Y_1 \sim \pi_\tau$.
Thus, $f(\hat Y_\tau)$ and $f(Y_0)$ share the same law so that 
$\ee^{\pi_\tau} [\ee [f(\hat Y_\tau^x)-f(Y_0)]|_{x=\PPP(Y_0)}]=0$. 
	Then \eqref{Atau} implies 
	\begin{align} \label{e-atau}
	 |\ee^{\pi_\tau} \mathcal A_\tau f_\varphi(\PPP(Y_0))| 
	& =|\ee^{\pi_\tau}  \ee [f(Y_0)-f(x)]|_{x=\PPP(Y_0)}]| \nonumber \\
	& \le \|\nabla f\| \ee^{\pi_\tau} |Y_0-\PPP(Y_0)|.
	\end{align}
 
Let $y=\PPP(x)$.
By the first two identities of the representation \eqref{R}, It\^o formula, and the identity \eqref{af}, we have  
\begin{align*}
	& |\ee \mathcal R_1(x, g_\tau(x))| \\
	&=\Big|\int_0^\tau \ee \int_0^s \Big[\nabla^2 f(\hat Y_r^{y}) (b_\tau(\PPP(x)), b_\tau(\PPP(x))) \\
	& \quad + \frac{1}{2} \sum_{j=1}^m \nabla^3 f(\hat Y_r^y) (b_\tau(\PPP(x)),  \sigma_{j, \tau}(\PPP(x)), \sigma_{j, \tau}(\PPP(x))) \Big] \dif r \dif s\Big|\\
	&\le \frac14 (2 \Vert \nabla^2 f\Vert |b_\tau(\PPP(x))|^2
	+ \Vert \nabla^3 f \Vert \| \sigma_\tau(\PPP(x))\|_{HS}^2 \vert b_\tau(\PPP(x))\vert) \tau^2, \\
	& |\ee \mathcal R_2(x, g_\tau(x))| \\
	&=\frac{1}{2} \Big|\int_0^\tau \sum_{j=1}^m \ee \Big[ \int_0^s\nabla^3 f(\hat Y_r^y) (b_\tau(x), \sigma_{j, \tau}(x), \sigma_{j, \tau}(x)) \\
	& \quad + \frac{1}{2} \sum_{k=1}^m \nabla^4 f(\hat Y_r^y) (\sigma_{j, \tau}(\PPP(x)),  \sigma_{j, \tau}(\PPP(x)), \sigma_{k, \tau}(x), \sigma_{k, \tau}(\PPP(x))) \dif r \Big] \dif s\Big|\\
	&\le \frac18 (2 \Vert \nabla^3 f \Vert \| \sigma_\tau(\PPP(x))\|_{HS}^2 |b_\tau(\PPP(x))|
	+ \| \nabla^4 f \| \| \sigma_\tau(\PPP(x))\|_{HS}^4)\tau^2.
\end{align*}
Similarly, the last four identities of the representation \eqref{R} imply that
\begin{align*} 
	|\ee \mathcal R_3(x, g_\tau(x))| 
	& \le \|\nabla^2 f\| |\PPP(x)-x| |b_\tau(\PPP(x))| \tau, \\ 
	|\ee \mathcal R_4(x, g_\tau(x))| 
	& \le \|\nabla^3 f\|  |\PPP(x)-x| \|\sigma_\tau(\PPP(x))\|_{HS}^2 \tau, \\
	|\ee \mathcal R_5(x, g_\tau(x))| 
	& \le \|\nabla f\| |b_\tau(\PPP(x))-b(x)| \tau, \\ 
	|\ee \mathcal R_6(x, g_\tau(x))| 
	& \le \|\nabla^3 f\| \sum_{j=1}^m |\sigma_{j, \tau}(\PPP(x))-\sigma_j(x)| \tau.   
\end{align*} 

In combination the estimates in \eqref{est-bem-} with Assumption \ref{A2} and the estimate \eqref{e-atau}, we obtain  
\begin{align} \label{est-R-mem}
\Big|\sum_{i=1}^6 \ee ^{\pi_\tau} [\ee \mathcal R_i(x, g_\tau(x))|_{x=Y_0}] \Big| 
& \le C \|\nabla^4 f\| (1 + \pi_\tau( |\cdot|^{\alpha_1}))  \tau^2, \\
\label{e-atau+}
	 |\ee^{\pi_\tau} \mathcal A_\tau f_\varphi(\PPP(Y_0))| 
	& \le C \|\nabla f\| \pi_\tau( |\cdot|^{\alpha_2}) \tau^2.
	\end{align}
	Then we conclude \eqref{err-tau} by the above two estimates \eqref{est-R-mem} and \eqref{e-atau+}, together with the estimate \eqref{err}. 
	
	Finally, we need to remark that $\pi_\tau(|\cdot|^{\max\{\alpha_1, \alpha_2\}}) <\infty$, as $\pi_\tau$ possesses all bounded $2p$-moments for any $p \in [1,p^{\star}] \cap \nn_+$, which follows from the moments estimate of the MEM \eqref{MEM} in \cite[Theorem 3.2]{LWWZ25}.

(2) For the BEM \eqref{bem}, the associated intepolation process $(\hat Y_t)_{t\in[0,\tau]}$ can be choosen as 
\begin{align}
\hat Y_t=\hat Y_0+\int_0^t b(Y_0)\dif s+\int_0^t \sigma(Y_0)\dif W(s), \quad t \in [0, \tau],
	\end{align}
i.e., $g_\tau(x):=x-b(x) \tau$ and $\tilde g_\tau(x)=x$, $x \in \rr^d$.  
	It is straightforward to verify that $\hat Y_\tau=Y_1-b (Y_1)\tau=g_\tau(Y_1)$, and that $\hat Y$ is continuous so that it is an It\^o process.  
	
	Let $Y_0 \sim \pi_\tau$. As $\pi_\tau$ is an invariant measure of \eqref{y}, we have $Y_1 \sim \pi_\tau$.
	Since $\hat Y_0=g_\tau(Y_0)$ and $\hat Y_\tau=g_\tau(Y_1)$, $\hat Y_0$ and $\hat Y_\tau$ have the same law. 
	Then \eqref{Atau} implies $\ee^{\pi_\tau} \mathcal A_\tau f_\varphi(\hat Y_0)=0$.

Similarly to the above estimates of $\ee \mathcal R_i(x, g_\tau(x))$ with $i=1,2,3,4$ for the MEM \eqref{MEM}, we have  
\begin{align*}
|\ee \mathcal R_1(x, g_\tau(x))|
&\le \frac14 (2 \Vert \nabla^2 f\Vert |b(x)|^2
+ \Vert \nabla^3 f \Vert \| \sigma(x)\|_{HS}^2 \vert b(x)\vert) \tau^2, \\
|\ee \mathcal R_2(x, g_\tau(x))|
& \le \frac18 (2 \Vert \nabla^3 f \Vert \| \sigma(x)\|_{HS}^2 |b(x)|
+ \| \nabla^4 f \| \| \sigma(x)\|_{HS}^4)\tau^2 \\
|\ee \mathcal R_3(x, g_\tau(x))| 
& \le \|\nabla^2 f\| |b(x)|^2 \tau^2, \\ 
|\ee \mathcal R_4(x, g_\tau(x))| 
& \le \|\nabla^3 f\| |b(x)| \|\sigma(x)\|_{HS}^2 \tau^2. 
\end{align*}
Since $\tilde g_\tau(x)=x$, $x \in \rr^d$, and there is no tame for the BEM \eqref{bem}, it is clear that 
\begin{align*}
|\ee \mathcal R_5(x)| = |\ee \mathcal R_6(x)|=0.
\end{align*}
Consequently, 
\begin{align} \label{est-R-bem}
\Big|\sum_{i=1}^6 \ee \mathcal R_i(x)\Big| 
& \le  C (1 + |b(x)|^2+|b(x)| \|\sigma(x)\|_{HS}^2+\|\sigma(x)\|_{HS}^4)  \tau^2 \nonumber \\
& \le C (1 +  |x|^{2(\gamma+1)})  \tau^2.
\end{align}  
In combination with the estimate \eqref{est-R-bem} and the fact $\ee^{\pi_\tau} \mathcal A_\tau f_\varphi(\hat Y_0)=0$, we conclude \eqref{err-tau}. 

Finally, we mention that $\pi_\tau(|\cdot|^{2(\gamma+1)}) <\infty$, as $\pi_\tau$ possesses all bounded $2p$-moments for any $p \in [1,p^{\star}] \cap \nn_+$, which follows from the moments estimate of the BEM \eqref{bem} in \cite[Lemma 10]{LW25}.

(3) The error estimate \eqref{err-N} of the ergodic limit follows from the $V$-geometrical ergodicity of these one-step approximations in Remark \ref{rk} and the error \eqref{lim}.
\end{proof}

\begin{rk} 
Using the same arguments as in the proof of Theorem \ref{tm-sde}, we obtain the following sharp ergodic error estimates for the EM scheme with a unique invariant measure $\pi_\tau$, under Lipschitz conditions as in \cite{FSX19, JSS23, PP23}: 
		\begin{align*}  
\sup_{\|\varphi\|_{\CC_b^4} \le 1} |\pi_\tau(\varphi)-\pi(\varphi)| \le C \tau, \quad 
\Big|\frac{1}{N}\sum_{k=0}^{N-1}\mathbb E\varphi(Y_k^x)-\pi(\varphi) \Big|
 \le C(1+|x|^2) (\tau+N^{-1}).
		\end{align*}
\end{rk}

\begin{rk}
If one can derive the regularity estimate \eqref{est-fp} of the Stein equation solution for any $\varphi \in \BB_b(\rr^d)$ or ${\rm Lip}(\rr^d)$, then the ergodic error estimates in Theorem \eqref{tm-sde} hold under total variation distance and Wasserstein distance, respectively.
\end{rk}

\section{Numerical Experiments}
\label{sec4}

This section presents comprehensive numerical investigations to validate the ergodicity (via the strong mixing property) and the ergodic estimates for the BEM \eqref{bem} and the TEM \eqref{tem}. We consider the following coefficients:
\begin{align*}
    b(x) = -x - x^3, \qquad \sigma(x) = \frac{1}{2}\sqrt{x^{2}+1}, \quad x \in \mathbb{R}.
\end{align*}

\subsection{Verification of Assumptions}

We first verify that the chosen coefficients satisfy all assumptions required by our theoretical framework.

For any $x, y \in \mathbb{R}$, direct computation yields
\begin{align*}
    \< x - y, b(x) - b(y) \> &= (x-y) [(-x-x^{3}) - (-y-y^{3})] \\
    &= -(x-y)^{2} - (x-y)^{2}(x^{2} + xy + y^{2}) \le -(x-y)^{2}.
\end{align*} 
Moreover, using the Lipschitz continuity of the square-root function,
\begin{align*}
    |\sigma(x) - \sigma(y)|^{2} = \frac{1}{4}\Big|\sqrt{x^{2}+1} - \sqrt{y^{2}+1}\Big|^{2} \le \frac{1}{4}|x-y|^{2}.
\end{align*}
Consequently, for any $p^{\star} \ge 2$, we obtain
\begin{align*}
    \< x-y, b(x)-b(y) \> + \frac{2p^{\star}-1}{2}|\sigma(x)-\sigma(y)|^{2}
    \le -\Bigl[1 - \frac{2p^{\star}-1}{8}\Bigr] |x-y|^{2}.
\end{align*}
Choosing $L_1 = 1 - (2p^{\star}-1)/8 > 0$ (e.g., $L_1 = 5/8$ when $p^{\star}=2$) ensures the monotonicity condition \eqref{mon} in Assumption \ref{A1} holds.

Regarding the dissipative condition \eqref{coe}, we have
\begin{align*}
    \< x, b(x) \> + \frac{p^{\star}(2p^{\star}-1)}{2}|\sigma(x)|^{2}
    &= -x^{2} - x^{4} + \frac{p^{\star}(2p^{\star}-1)}{8}(x^{2}+1) \\
    &\le L_2 - L_3 |x|^{4},
\end{align*}
where one may take $\gamma = 3$ and select suitable positive constants $L_2, L_3$. 
Since $b$ is cubic and $\sigma$ is smooth with derivatives of polynomial growth, Assumption \ref{A2} is also satisfied. Hence, the coefficient pair $(b,\sigma)$ is admissible under our theoretical setting.

\subsection{Strong Mixing Property}

To examine the strong mixing property (which implies unique ergodicity) of TEM \eqref{tem} and the BEM \eqref{bem}, we fix the step-size $\tau = 0.2$ and simulate $N = 5000$ steps, corresponding to a final time $T = 1000$. Three distinct initial values, $X_0 = -5, \;5, \;15$, are used. For each scheme, $M = 5000$ independent sample paths are generated. 

\begin{figure}  \hspace{-4em}
	\subfigure{
		\begin{minipage}[t]{0.475\linewidth}
			\includegraphics[width=7.5cm]{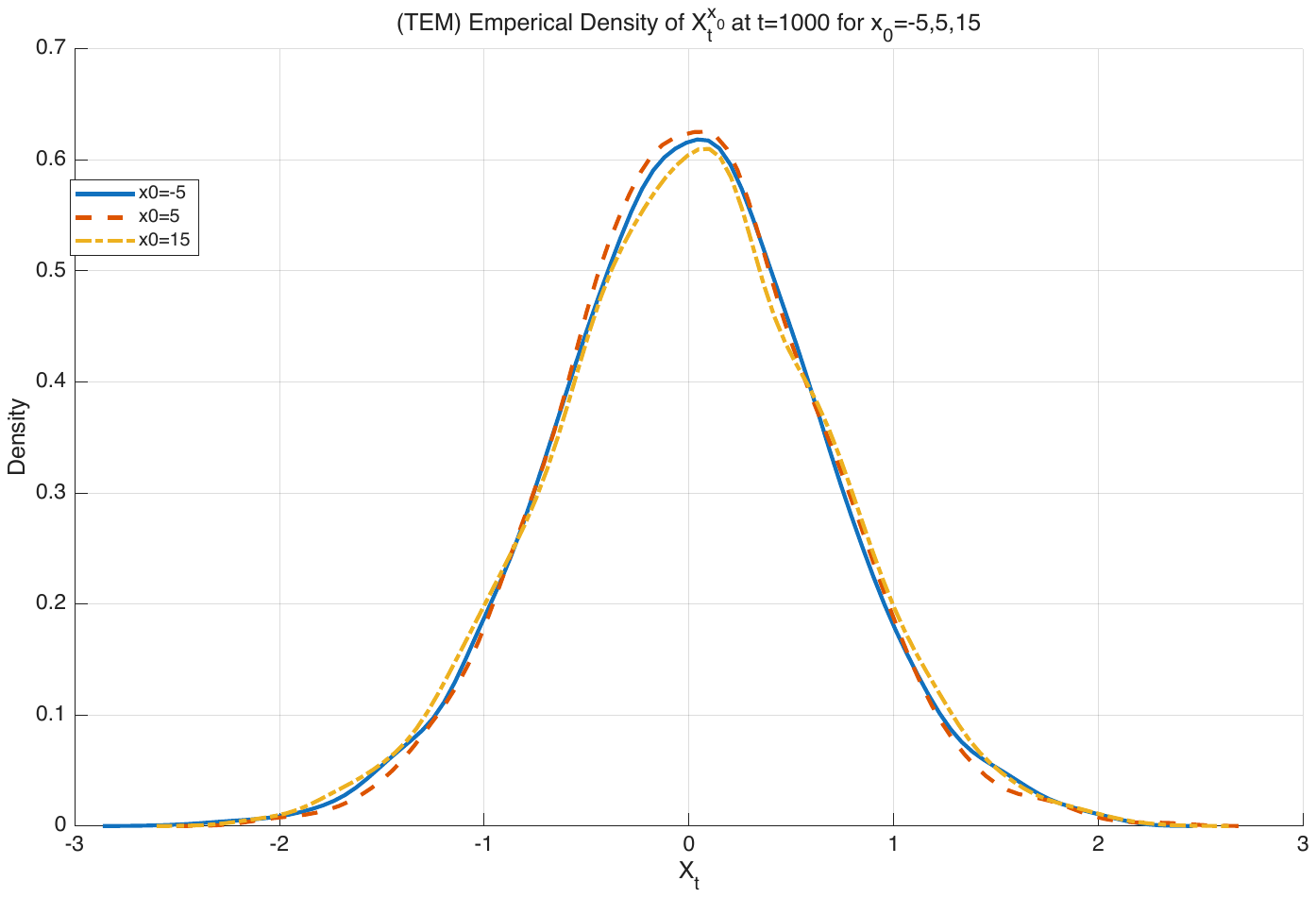}
		\end{minipage} } 
	\subfigure{
		\begin{minipage}[t]{0.475\linewidth}
			\includegraphics[width=7.5cm]{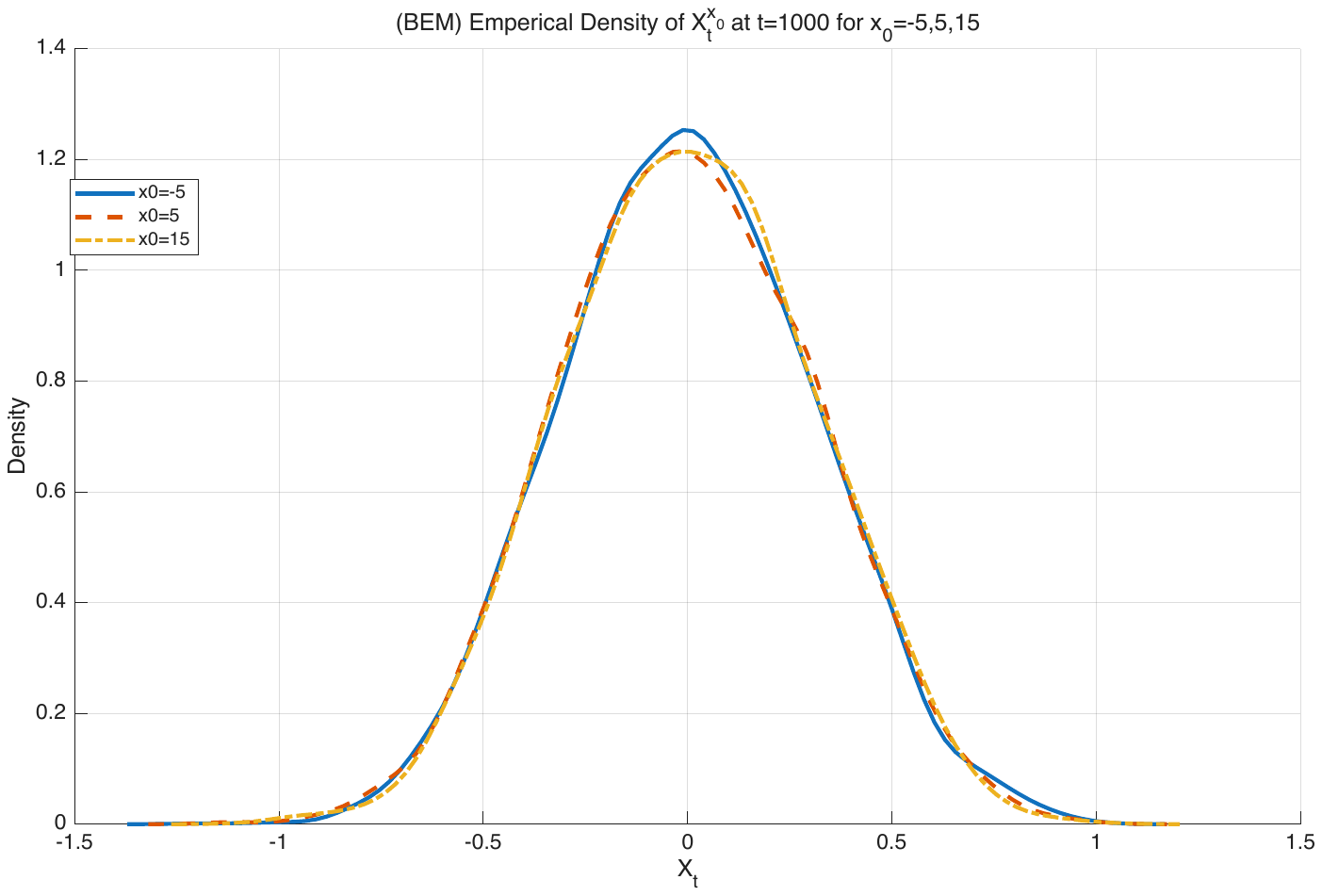}
		\end{minipage} }
	\caption{Empirical density functions of TEM \eqref{tem} (left) and spatial BEM \eqref{bem} (right).}\label{density}
\end{figure} 

The empirical density at the terminal time $T = 1000$ is reconstructed via Gaussian kernel density estimation. 
Figure \ref{density} displays the resulting empirical density functions for the TEM \eqref{tem} (left) and the BEM \eqref{bem} (right), respectively. Despite the disparate initial conditions, the density curves for each scheme exhibit nearly perfect overlap. This visual agreement numerically confirms that both discretizations possess the strong mixing property and, consequently, that their associated Markov chains are uniquely ergodic.

\subsection{Weak Convergence Order}

We now investigate the convergence order of the ergodic errors (i.e., weak convergence order) in Theorem \ref{tm-sde} of these two schemes. 
The test functions are chosen as $\varphi_1(x) = \cos(x)$ and $\varphi_2(x) = e^{-x^{2}}$, $x \in \mathbb R^d$, the initial value is fixed as $X_0 = 1$, and the simulation horizon is set to $T = 32$. A reference solution is computed using a fine step-size $\tau_{\text{ref}} = 2^{-11}$. Comparative simulations are performed with coarser step-sizes $\tau_i = 2^{-i}$ for $i = 3,4,5,6,7$. Expectations are approximated via Monte Carlo averages over $M = 20000$ independent sample paths.
\begin{figure} \hspace{-3em}
	\subfigure{
		\begin{minipage}[t]{0.475\linewidth}
			\includegraphics[width=7.5cm]{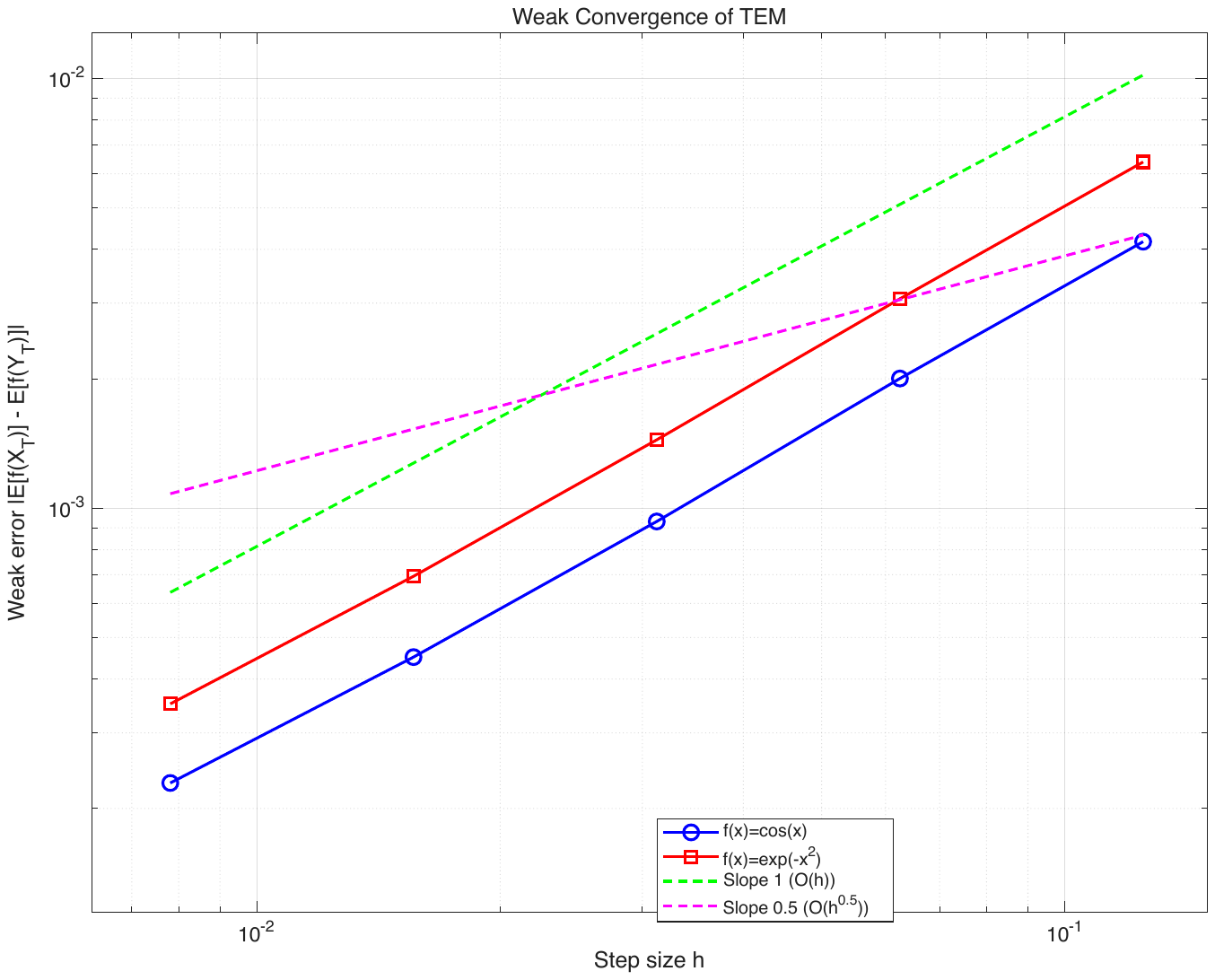}
		\end{minipage} }
	\subfigure{
		\begin{minipage}[t]{0.475\linewidth}
			\includegraphics[width=7.5cm]{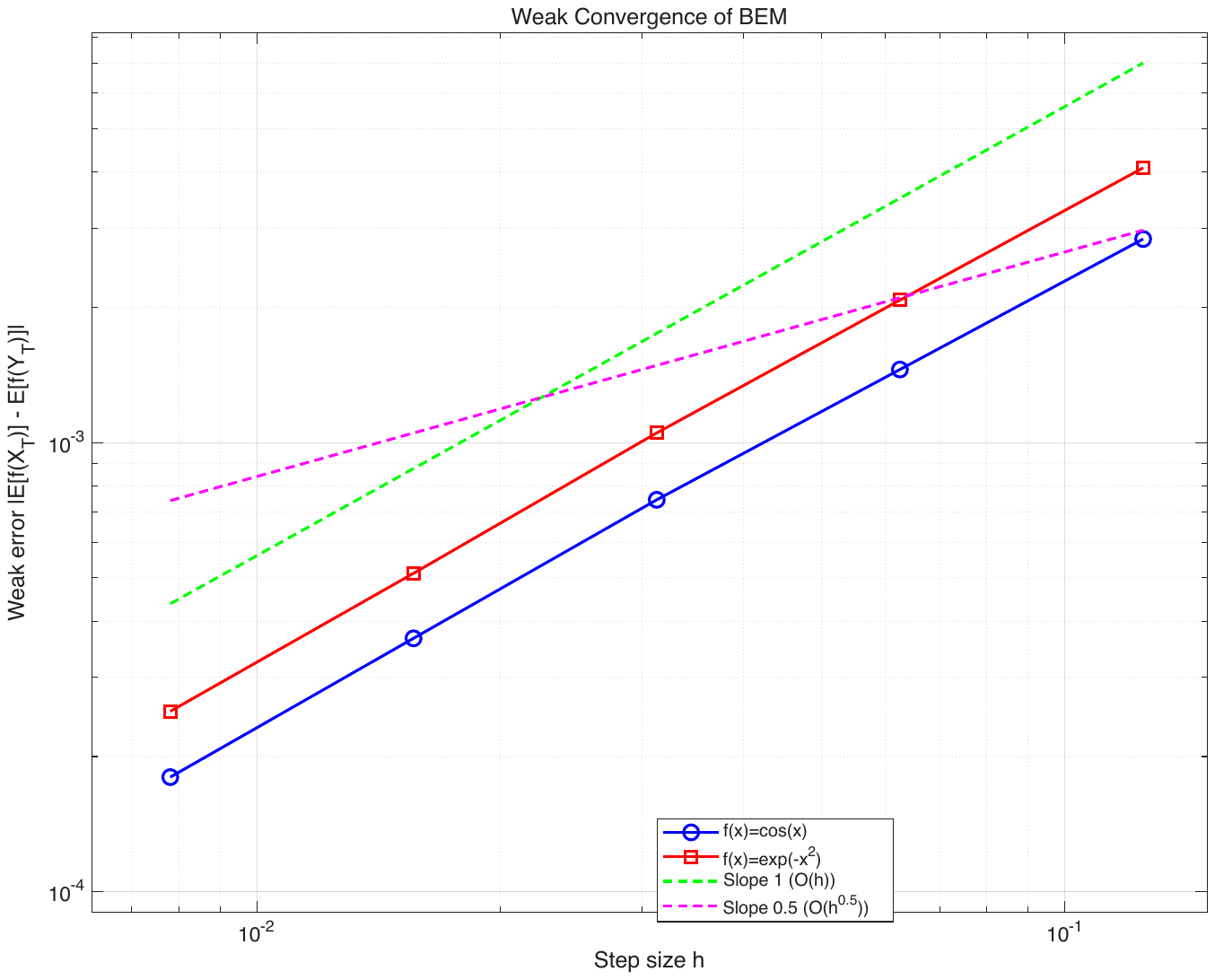}
		\end{minipage} }
	\caption{Weak convergence orders of TEM \eqref{tem} (left) and spatial BEM \eqref{bem} (right).}\label{order}
\end{figure} 

Figure \ref{order} presents, in double logarithmic coordinates, the weak errors of the TEM \eqref{tem} and the BEM \eqref{bem} as functions of $\tau$. For TEM, the estimated convergence orders are $1.07$ with $\varphi_1(x) = \cos(x)$ and $1.08$ with $\varphi_2(x) = e^{-x^2}$. For BEM, the corresponding estimates are $1.08$ and $1.02$. These numerical results robustly support the theoretical prediction in Theorem \ref{tm-sde}, which shows that both schemes achieve a first-order ergodic error for SODEs with superlinear coefficients driven by multiplicative noise.

\bibliographystyle{abbrv}
\bibliography{bib}
        
\end{document}